\documentclass[10pt]{amsart}
\usepackage{version,amsmath, amssymb}
\usepackage{color}
\usepackage{amsthm, amsopn, amsfonts, xypic}
\usepackage[colorlinks=true]{hyperref}
\hypersetup{urlcolor=blue, citecolor=blue}

\let\arXiv\arxiv

\def\doi#1{ {\href{http://dx.doi.org/#1}
   {{\mdseries\ttfamily DOI}}}}

\let\MR\mr

\newcommand{\cal}{\mathcal}

\newtheorem{thm}{Theorem}[section]
\newtheorem{prop}[thm]{Proposition}
\newtheorem{coro}[thm]{Corollary}
\newtheorem{rem}{Remark}[section]

\newtheorem{lem}[thm]{Lemma}

\newcommand{\cd}{\, \cdot\, }

\newcommand{\R}{{\mathbb R}}

\newcommand{\ang}{{\not\negmedspace\nabla}}

\newcommand{\la}{\langle}
\newcommand{\ra}{\rangle}

\newcommand{\rs}{{r^*}}

\renewcommand{\S}{{\mathbb S}}
\newcommand{\M}{\cal M}
\newcommand{\tv}{{\tilde{v}}}

\newcommand{\al}{\alpha}    
\newcommand{\de}{\delta}    
\newcommand{\ep}{\epsilon}  

\newcommand{\ga}{\gamma}    
\def\<{\langle}             \def\>{\rangle}

\newcommand{\pt}{\partial_t}\newcommand{\pa}{\partial}

\newcommand{\beeq}{\begin{equation}}\newcommand{\eneq}{\end{equation}}

\newenvironment{prf}{\noindent {\bf Proof.} }{\endprf\par}
\def \endprf{\hfill  {\vrule height6pt width6pt depth0pt}\medskip}

\def\CO{\mathcal {O}}

\numberwithin{equation}{section}

\newcommand{\weight}{{\Bigl(1-\frac{2M}{r}\Bigr)}}

\begin{document}

\title{The Strauss conjecture on Kerr black hole backgrounds}

\author[Lindblad]{Hans Lindblad}
\address{Department of Mathematics, Johns Hopkins University, Baltimore,
MD 21218, USA} \email{lindblad@math.jhu.edu}
\urladdr{http://www.math.jhu.edu/~lindblad}

\author[Metcalfe]{Jason Metcalfe}
\address{Department of Mathematics, University of North Carolina,
  Chapel Hill, NC  27599-3250, USA}
\email{metcalfe@email.unc.edu}
\urladdr{http://www.unc.edu/~metcalfe}

\author[Sogge]{Christopher D. Sogge}
\address{Department of Mathematics, Johns Hopkins University, Baltimore,
MD 21218, USA} \email{sogge@jhu.edu}
\urladdr{http://www.mathematics.jhu.edu/sogge}

\author[Tohaneanu]{Mihai Tohaneanu}
\address{Department of Mathematics, Johns Hopkins University, Baltimore,
MD 21218, USA} \email{mtohanea@math.jhu.edu}
\urladdr{http://www.math.jhu.edu/~mtohanea}

\author[Wang]{Chengbo Wang}
\address{Department of Mathematics, Zhejiang University, Hangzhou 310027, China}
\email{wangcbo@gmail.com}
\urladdr{http://www.math.zju.edu.cn/wang}

\thanks{The first three authors were supported in part by the NSF}

\thanks{
The fifth author was supported by Zhejiang Provincial Natural Science Foundation of
China LR12A01002, the Fundamental Research Funds for the Central Universities, NSFC 11271322 and J1210038.}


\begin{abstract}We examine solutions to semilinear wave equations on
  black hole backgrounds and give a proof of an analog of the Strauss conjecture
on the Schwarzschild and Kerr, with small angular momentum, black hole
backgrounds.  The key estimates are a class of weighted
Strichartz estimates, which are used near infinity where the metrics
can be viewed as small perturbations of the Minkowski metric, and a
localized energy estimate
on the black hole background, which handles the behavior in the
remaining compact set.
\end{abstract}

\maketitle


\section{Introduction}

In this article, we study an analog of the Strauss conjecture on the
Schwarzschild and Kerr, with small angular momentum, black hole
backgrounds.  In particular, we establish the global existence of
solutions to a class of semilinear wave equations with power-type
nonlinearities with power greater than a certain critical power.  This
critical power, $1+\sqrt{2}$, is the same as that on
$(1+3)$-dimensional Minkowski space.

More specifically, we will consider the evolution of the nonlinear
waves on Kerr black hole backgrounds,
\begin{equation}\label{nlw}
\Box_K u = F_p(u), \qquad u |_{\tv=0} = f, \qquad \tilde T u |_{\tv=0} = g\ .
\end{equation}
Here $\Box_K$ denotes the d'Alembertian in the Kerr metric, and
$\tilde T$ is a smooth, everywhere timelike vector field that equals
$\pa_t$ away from the black hole.  Similarly, the coordinate $\tv$
is chosen so that the slice $\tv=0$ is space-like and so that $\tv=t$
away from the black hole.  A more
detailed description of the Kerr geometry is provided in the next
section.
We shall assume that the nonlinear term behaves like $|u|^p$ when $u$
is small:
\beeq\label{key}
 \sum_{0\leq j\leq 2} |u|^j |\partial_u^j F_p(u)| \lesssim |u|^p\ \textrm{ for } |u|\ll 1\ .
\eneq
Typical examples include $F_p(u)=\pm |u|^p$ and $\pm |u|^{p-1}u$.

The Strauss conjecture concerned the Minkowski case and determining
the values $p$ for which global existence could be guaranteed if the
initial data are sufficiently small.  The first work \cite{John79}
showed that small data global existence was available in
$(1+3)$-dimensional Minkowski space-time for powers $p>1+\sqrt{2}$ but
blow-up could occur for arbitrarily small data when $p<1+\sqrt{2}$.
Shortly afterward, \cite{Strauss81} included the conjecture that the critical power $p_c$ on Minkowski space-time
$\R^{n+1}$ is the positive root of the quadratic equation $$(n - 1)p^2
- (n + 1)p - 2 = 0\ .$$
The existence portion of the conjecture was verified in \cite{G}
($n=2$), \cite{Zh95} ($n=4$), \cite{LdSo96} ($n\le 8$), and
\cite{GLS97}, \cite{Ta01-2} (generic $n$).  The necessity of $p>p_c$
for small data global existence is from \cite{John79}, \cite{Glassey},
\cite{Sideris}, \cite{Schaeffer}, \cite{YorZh06}, and \cite{Zh07}.

Some recent works have sought to extend these results to scenarios
that include nontrivial background geometry.  These include
\cite{DMSZ} ($n=4$), \cite{HMSSZ} ($n=3,4$), and \cite{SmSoWa12}
($n=2$), which examine global existence for similar equations exterior
to nontrapping obstacles.  See also \cite{Yu11} ($n=3,4$) for related
results with certain trapping obstacles.  In the case of nontrapping
aymptotically Euclidean manifolds the same results were obtained in
\cite{SoWa10} (radial metrics, $n=3$) and \cite{WaYu11} (general
metrics $n=3,4$).

We seek to show the same on Kerr black hole backgrounds with small
angular momentum.  In particular, we have
\begin{thm}\label{metaTheorem}
For Kerr space-times with sufficiently small angular momentum and for initial data which are smooth, compactly supported, and
sufficiently small, there exists a global solution $u$ to \eqref{nlw}
provided that $p>1+\sqrt{2}$.
\end{thm}
A more precisely stated theorem will be provided in Section
\ref{sectionExistence} after more notation is introduced.  In fact,
our theorem holds more generally than we state.  The proof does not
rely on the precise geometry of the Kerr space-time and rather only
depends on having a metric which is asymptotically Euclidean and for
which there is a sufficiently nice localized energy estimate.  The
latter will be described further in the next section.

On the Schwarzschild space-times, such nonlinear wave equations have
been previously studied for large powers.  In particular, see
\cite{BN93} and \cite{Ni95} for related Klein-Gordon equations,
\cite{DaRo05} ($p>4$ with radial data), \cite{BlSt06} ($p>3$), and
\cite{MMTT} ($p=5$), though well-known arguments (see, e.g., \cite{So08}) allow one to use
Strichartz estimates, such as those proved in \cite{MMTT}, to prove
small data global existence for other $p>3$.  While no such explicit
results have been previously given on Kerr backgrounds, the key
estimates are known in some cases.  For example, \cite{To12} provides
Strichartz estimates in the case that the angular momentum is small.
In the opposite direction, \cite{CG06} provides blow-up for
$p<1+\sqrt{2}$.  The current result fills in the gap $1+\sqrt{2}<p<3$.

The strategy of proof is to use the weighted Strichartz estimates of
\cite{FaWa} and \cite{HMSSZ} (see also \cite{Hi07} for a radial
version, which appeared previously) near infinity where the Kerr
metric can be viewed as a small perturbation of the Minkowski metric.
See \cite{So08} and \cite{WaYu11p} for more details on the history.
In the compact set that remains, localized energy estimates suffice.
We rely on the localized energy estimates of \cite{TT}, though other
variants of these are available as will be described in the following section.

In the case of the Kerr metric our proof requires that the initial
data have compact support.  However, as we shall see in Section
\ref{SchwSection}, we are able to drop this technical assumption in
the special case of the Schwarzschild metric.
It will be interesting to see whether such results hold for the Kerr metric as well.

For convenience and to ease the exposition, we have taken data on a
$\tv=0$ slice.  Passing from data on $t=0$ to data on $\tv=0$ is a
problem of local well-posedness, which can be solved by contraction in
energy spaces.  We omit these details, though a related argument, which
requires Strichartz estimates, can
be found in \cite{MMTT}.

\subsection{Notation}
The relevant sets of vector fields we shall use are as follows
$$
\{\pa_t, \nabla_x\}=\{\pa\},\quad \Omega =x\wedge \nabla_x\ ,
$$
$$Y=\{\nabla_x, \Omega\},\quad Z=\{\pa, \Omega\}=\{\pt\}\cup Y\ .$$
Let $\<x\>=\sqrt{1+|x|^2}$ and $L^q_\omega$ be the standard Lebesgue space on the sphere $\S^2$.
We will use the following mixed-norm $L^{q_1}_\tv L^{q_2}_r L^{q_3}_\omega$,
$$\|f\|_{L^{q_1}_\tv L^{q_2}_r L^{q_3}_\omega(\M)}=\left\|\left(\int_{r_e}^\infty \|f(\tv, r\omega)\|_{L^{q_3}_\omega}^{q_2} r^2 dr\right)^{1/q_2}\right\|_{L^{q_1}(\{\tv\ge 0\})},$$ with trivial modification for the case $q_2=\infty$. Occasionally, we will omit the subscripts. We will also use
$A \lesssim B$ to denote the inequality $A\le C B$ with some positive constant
$C$, which may change from line to line.  We also use the following
convention (for invertible functions $f$ and function spaces $H$)
$$g\in f H\Leftrightarrow f^{-1}g\in H\ .$$

Let us also recall some notations from \cite[Section 5.6]{JJ}.
Let $A$ be a Banach space, $s\in\R$ and $q>0$, we use $l^s_q(A)$ to denote the space of all sequences $(a_j)_{j=0}^\infty$, $a_j\in A$ such that
$$\|(a_j)\|_{l^s_q(A)}=\|2^{js}\|a_j\|_A\|_{l^q_{j\ge 0}}<\infty\ .$$
For a partition of unity subordinate to the dyadic (spatial) annuli,
$1=\sum_{j\ge 0}\phi^2_j(x)$, we shall abuse notation and write
\[
 \|u\|_{l^s_q(A)} = \|(\phi_j(x)u(t, x))\|_{l^s_q(A)}.
\]


\section{The Kerr metric and localized energy estimates}

Let us first recall the Kerr metric.  In Boyer-Lindquist coordinates, it is given by
\[
ds^2 = g_{tt}\,dt^2 + g_{t\phi}\,dt\,d\phi + g_{rr}\,dr^2 +
g_{\phi\phi}\,d\phi^2 + g_{\theta\theta}\,d\theta^2
\]
where $t \in \R$, $r > 0$, $(\phi,\theta)$ are the spherical
coordinates on $\S^2$ and
\[
g_{tt}=-\frac{\Delta-a^2\sin^2\theta}{\rho^2}, \qquad
g_{t\phi}=-2a\frac{2Mr\sin^2\theta}{\rho^2}, \qquad
g_{rr}=\frac{\rho^2}{\Delta},
\]
\[ g_{\phi\phi}=\frac{(r^2+a^2)^2-a^2\Delta
  \sin^2\theta}{\rho^2}\sin^2\theta, \qquad g_{\theta\theta}={\rho^2}
\]
with
\[
\Delta=r^2-2Mr+a^2, \qquad \rho^2=r^2+a^2\cos^2\theta.
\]
Here $M$ represents the mass of the black hole and $aM$ its angular
momentum.  The Schwarzschild space-time is the static solution
corresponding to $a=0$.  And the Minkowski space-time is the trivial
solution to Einstein's equations for which we have $a=0$ and $M=0$.

For convenience, we record that
\[d\text{Vol} = \rho^2 \sin\theta\,dr\,d\theta\,d\phi\,dt = \rho^2
\,dr\,d\omega\,dt\]
where $\omega$ is the standard measure on $\S^2$.  Moreover,
the inverse of the metric is given by:
\[ g^{tt}=-\frac{(r^2+a^2)^2-a^2\Delta\sin^2\theta}{\rho^2\Delta},
\qquad g^{t\phi}=-a\frac{2Mr}{\rho^2\Delta}, \qquad
g^{rr}=\frac{\Delta}{\rho^2},
\]
\[
g^{\phi\phi}=\frac{\Delta-a^2\sin^2\theta}{\rho^2\Delta\sin^2\theta} ,
\qquad g^{\theta\theta}=\frac{1}{\rho^2}.
\]

One can
view $M$ as a scaling parameter, and $a$ scales in the same way as
$M$. Thus $M/a$ is a dimensionless parameter.  We shall subsequently
assume that $a$ is small, $a/M \ll 1$, so that the Kerr metric is a
small perturbation of the Schwarzschild metric.  Provided $M\neq 0$, one could set $M
= 1$ by scaling, but we prefer to keep $M$ in our formulas.  We let
$g_S$, $g_K$ denote the Schwarzschild, respectively
Kerr, metric, and $\Box_S$, $\Box_K$ denote the
associated d'Alembertians, where the (scalar) d'Alembertian is given by $\Box =
\nabla^\gamma \partial_\gamma$ with $\nabla$ denoting the metric
connection.

The Kerr metric has a singularity at $r = 0$
on the equator $\theta = \pi/2$.  The apparent singularities at the
roots of $\Delta$, namely at the horizons
$r=r_{\pm}:=M\pm\sqrt{M^2-a^2}$, are merely coordinate singularities.  For a
further discussion of the nature of $r_\pm$, which is not relevant for our
results, we refer the reader to, e.g., \cite{Ch},\cite{HE}.

To remove the coordinate singularities at $r=r_{\pm}$, we may
introduce Eddington-Finkelstein coordinates.  See, e.g., \cite{HE}.  To do so, we let $r^*$,
$v_+$, and $\phi_+$ solve
\[
dr^*=(r^2+a^2)\Delta^{-1}dr, \qquad dv_{+}=dt+dr^*, \qquad
d\phi_{+}=d\phi+a\Delta^{-1}dr.
\]
The metric then takes the form
\[
\begin{split}
  ds^2= &\
  -\Bigl(1-\frac{2Mr}{\rho^2}\Bigr)\,dv_{+}^2+2\,dr\,dv_{+}-4a\rho^{-2}Mr\sin^2\theta
  \,dv_{+}\,d\phi_{+} -2a\sin^2\theta \,dr\, d\phi_{+}\\
  & \ +\rho^2 d\theta^2  +\rho^{-2}\Bigl[(r^2+a^2)^2-\Delta a^2\sin^2\theta\Bigr]
\sin^2\theta \,d\phi_{+}^2,
\end{split}
\]
which is nondegenerate up to the metric singularity at $\rho=0$.

For our purposes, the Boyer-Lindquist coordinates are
convenient at spatial infinity but not near the event horizon, while
the Eddington-Finkelstein coordinates are convenient near the event
horizon but not at spatial infinity. To combine the two
we replace the $(t,\phi)$ coordinates with $(\tv,\phi_+)$,
as in \cite{MMTT} and \cite{TT}, by defining
\[
\tv = v_{+} - \mu(r)
\]
where $\mu$ is a smooth function of $r$. In these $(\tv,r,\phi_{+},
\theta)$ coordinates the metric has the form
\begin{multline*}
  ds^2=  \Bigl(1-\frac{2Mr}{\rho^2}\Bigr)\, d\tv^2
  +2\left(1-\Bigl(1-\frac{2Mr}{\rho^2}\Bigr)\mu'(r)\right) \,d\tv\, dr
   -4a\rho^{-2}Mr\sin^2\theta \,d\tv\, d\phi_{+}\\ + \left(2 \mu'(r) -
  \Bigl(1-\frac{2Mr}{\rho^2}\Bigr) (\mu'(r))^2\right) \, dr^2
  -2a\theta \Bigl(1+2\rho^{-2}Mr\mu' (r)\Bigr)\sin^2\theta \,dr\, d\phi_{+} +\rho^2\,
  d\theta^2
\\   +\rho^{-2}\Bigl[(r^2+a^2)^2-\Delta a^2\sin^2\theta\Bigr]\sin^2\theta\,
  d\phi_{+}^2.
\end{multline*}
The function $\mu$ is selected to satisfy:

(i) $\mu (r) \geq \rs$ for $r > 2M$, with equality for $r > {5M}/2$.

(ii) The surfaces $\tv = const$ are space-like, i.e.
\[
\mu'(r) > 0, \qquad 2 - (1-\frac{2Mr}{\rho^2}) \mu'(r) > 0.
\]

For $r_e$ fixed satisfying $r_-<r_e<r_+$, we shall consider the wave
equation
\begin{equation}
\Box_K u = F, \qquad u|_{\Sigma^-} = f, \qquad \tilde T u|_{\Sigma^-} = g\ ,
\label{CP}\end{equation}
in $\M =  \{ \tv \geq 0, \ r \geq r_e \}$ and with initial data on the
space like surface $ \Sigma^- =  \M \cap \{ \tv = 0 \}$.  The choice of $r_e$ is unimportant,
and for convenience we may simply use $r_e=M$ for all Kerr metrics
with $a/M\ll 1$.

We use $\ang$ to denote the angular derivatives
$\ang_i=\pa_i-\frac{x^i}{r} \pa_r$, where $x=r\omega$ is understood.  We set
\[
E[u](\Sigma^-) = \int_{\Sigma^-}
 \left(
|\partial_r u|^2 +   |\partial_\tv u|^2    +
|\ang u|^2 \right) r^2  \,dr\,  d\omega\ \]
to be the initial energy.  More generally, we use
\[
E[u](\tv_0) = \int_{ \M \cap \{\tv = \tv_0\}}
 \left(
|\partial_r u|^2 +   |\partial_\tv u|^2    +
|\ang u|^2 \right) r^2  \,dr \, d\omega
\]
to denote the energy on the space-like slice $\tv = constant$.  In
particular, $E[u](\Sigma^-) = E[u](0)$.


In Minkowski space, the localized energy estimate for the wave
equation states that
\[\|\partial u\|_{l^{-1/2}_\infty(L^2_tL^2_x)\cap L^\infty_tL^2_x} +
\|u\|_{l^{-3/2}_\infty(L^2_tL^2_x)} \lesssim \|\partial
u(0,\cd)\|_{L^2_x} + \|\Box u\|_{l^{1/2}_1(L^2_tL^2_x) +
  L^1_tL^2_x}.\]
Such estimates first appeared in \cite{Mo1,Mo2,Mo3} and subsequently
in, e.g., \cite{Strauss75}, \cite{KPV}, \cite{SmSo}, \cite{KSS1,
  KSS2}, \cite{BPST}, \cite{MetSo06, MetSo07}, \cite{Sterb}, and
\cite{HY}.  Such estimates are known to be fairly
robust and variants were proved in \cite{Al06}, \cite{MetSo06,
  MetSo07}, \cite{MetTa07}, \cite{Burq}, \cite{BoHa}, and \cite{SoWa10} for various
nontrapping metric perturbations.  The most basic proof of the
estimate above involves integrating $\Box u$ against $f(r)\partial_r u
+ \frac{n-1}{2}\frac{f(r)}{r}u$, where $f(r)=r/(r+2^j)$ and
integrating by parts.

The above proofs implicitly rely heavily on the fact that all null
geodesics escape to infinity.  In the case that there are trapped
rays, it is known \cite{Ral} that a loss is necessary in order to have a
variant of the localized energy estimate.

Members of the Kerr family of black holes contain trapped rays.  This is easiest
to describe in the Schwarzschild case where trapping occurs on the
event horizon $r=2M$ and on the so-called photon sphere $r=3M$.
Utilizing the red shift effect as in \cite{DaRo09} renders the trapping at
$r=2M$ inconsequential.  The known localized energy estimates on the
Schwarzschild space-time \cite{BS, BS2}, \cite{BlSt06}, \cite{DaRo,
  DaRo09}, and \cite{MMTT} reflect a loss at the
photon sphere.  These can be proved by choosing a related multiplier
where the $f$ is more complicated and switches sign at the photon
sphere.

The trapping on the Kerr space-times is more delicate and can only
be described in phase-space, though it does occur within an $O(a)$
neighborhood of $r=3M$.  See, e.g., \cite{TT}.  Since the region
containing trapped rays cannot be described only in physical space, it
is provable \cite{Al09} that no first order differential multipler, as used
above, can yield such a localized energy estimate.  Despite this,
there have been three related but distinct approaches that have
yielded localized energy estimates on Kerr backgrounds with small
angular momentum.  See \cite{AB09}, \cite{DaRoNotes}, and \cite{TT}.
See, also, \cite{DaRoNew, DaRoNew2} for the subextremal case $|a|< M$.

The approach that we shall follow is that of \cite{TT}.
We define our localized energy norm as
\begin{equation}
 \|u\|_{LE}  =  \|(1-\tilde\chi) \nabla u\|_{l^{-1/2}_\infty(L^2_{\tv,r,\omega})}
     + \|u\|_{l^{-3/2}_\infty(L^2_{\tv,r,\omega})}
 \label{leK}\end{equation}
where $\tilde\chi=\tilde\chi(r)$ is a smooth radial cut off function supported in
$[2.5M, 3.5M]$ and is identity in a neighborhood containing all of the
trapped rays.  It is this degeneracy of the norm that represents the
loss due to the trapping.

We then have the following
\begin{lem}\label{thm-tt}  Suppose $a\ll M$.
  Let $u$ solves the inhomogeneous wave equation $\Box_K u=F_1+F_2$
in the region $\M$ where $F_2$ is supported in $\{r\ge 3.5M\}$.  Then we have
\begin{equation}\label{main.estimate.inhom}
\sup_{\tilde v \geq 0}
    E[u](\tilde v)+  \|u\|_{LE}^2 \lesssim E[u](\Sigma^-) +
    \|F_1\|^2_{L^1_\tv L^2_{r,\omega}} +
\|F_2\|_{l^{1/2}_1(L^2_{\tv,r,\omega})}^2.
\end{equation}
\end{lem}
This is an easy corollary of \cite[Theorem 4.1]{TT}.  Indeed, the norm
$\|\,\cdot\,\|_{LE^1_K}$ defined therein satisfies
\[\|u\|_{LE} \lesssim \|u\|_{LE^1_K}\]
and, under the support conditions on $F_2$,
\[\|F_2\|_{l^{1/2}_1(L^2_{\tv,r,\omega})} \gtrsim \|F_2\|_{LE^*_K},\]
where the latter is defined in \cite{TT}.  Moreover, the control of
$F_1$ follows by bootstrapping \cite[(4.14)]{TT} using the energy term
rather than the localized energy norm.

We also have the following higher order version of
\eqref{main.estimate.inhom}, which similarly follows from
\cite[Theorem 4.5]{TT}.
\begin{coro}
  Let $n$ be a positive integer, and suppose that $u$ solves the inhomogeneous wave equation $\Box_K u=F_1+F_2$
in the region $\M$ where $F_2$ is supported in $\{r\ge 3.5M\}$.  Then we have
\begin{multline}\label{main.estimate.inhom.n}
\sup_{\tilde v \geq 0}
  \sum_{|\alpha|\le n}  E[\partial^\alpha u](\tilde v)+
  \sum_{|\alpha|\le n} \|\partial^\alpha u\|_{LE}^2 \\\lesssim
\|u(0,\cd)\|^2_{H^{n+1}_x} + \|\tilde T u(0,\cd)\|^2_{H^{n}_x}+
   \sum_{|\alpha|\le n} \|\partial^\alpha F_1\|^2_{L^1_\tv
     L^2_{r,\omega}} + \sum_{|\alpha|\le n}
\|\partial^\alpha F_2\|_{l^{1/2}_1(L^2_{\tv,r,\omega})}^2.
\end{multline}
\end{coro}

\section{Weighted Strichartz estimates}
In this section, we collect the required Sobolev-type and
weighted Strichartz estimates.

\subsection{Weighted Sobolev estimates}

In the sequel, we shall require the following weighted Sobolev
estimates.  These are straightforward variants of those that appeared
in, e.g., \cite{Kl1}.
\begin{lem}
For $R\ge 10$, $2\le p<\infty$, and any $b\in \R$, we have
\beeq\label{ap-S}
 \|r^b v\|_{L_r^{\frac{2p(p-1)}{p-2}}L_{\omega}^{\infty}(r\ge R+1)}
 \lesssim \sum_{|\gamma|\leq 2} \|r^{b-\frac{1}{p-1}} Y^\gamma
 v\|_{L_r^{p}L_{\omega}^{2}(r\ge R)},
\eneq
\beeq\label{ap-S4} \|r^b v \|_{L^{4}_x(|x|\ge R+1)}\lesssim
\sum_{|\gamma|\leq 1} \|r^{b+\frac{1}{2}-\frac{2}{p}} Y^\gamma v\|_{L_r^{p}L_{\omega}^{2}(r\ge
  R)}, \quad p\le 4,\eneq
and
\begin{equation}
  \label{ap-Sinfty}
  \|r^b v\|_{L^\infty_x(|x|\ge R+1)} \lesssim \sum_{|\gamma|\le 2}
  \|r^{b-\frac{2}{p}} Y^\gamma v\|_{L^p_r L^2_\omega(r\ge R)}.
\end{equation}
\end{lem}

\begin{prf}
By Sobolev's lemma on ${\mathbb R}\times \S^2$, we have for each $j\in {\mathbb N}$ the uniform bounds
$$\|v\|_{L^\infty_rL^\infty_\omega([j,j+1]\times S^2)}
\lesssim \sum_{|\gamma|\le 2}\Bigl(\int_{j-1}^{j+2}\int_{S^2}|Y^\gamma v|^2 \, d\omega dr\Bigr)^{\frac12}.
$$
Hence,
\beeq\label{ap-trace-local} \|v\|_{L_r^{\infty}([j, j+1])L_{\omega}^{\infty}} \lesssim j^{-1} \sum_{|\gamma|\leq 2} \|
Y^\gamma v\|_{L_r^{2}([j-1,j+2])L_{\omega}^{2}} .\eneq
Or more generally,
\[\|r^b v\|_{L^\infty_r([j,j+1])L^\infty_\omega} \lesssim
  \sum_{|\gamma|\le 2} \|r^{b-1} Y^\gamma v\|_{L^2_r([j-1,j+2])L^2_\omega}.\]
The factor $j^{-1}$ on the right comes from the fact that the volume
element for ${\mathbb R}^3$ is $r^2 dr d\omega$.  By H\"older's
inequality, we have that for every $1\le q<\infty$ and $p>2$
\beeq\label{ap-S3} \|r^b v\|_{L_r^{q}([j,j+1])L_{\omega}^{\infty}} \lesssim \sum_{|\gamma|\leq 2} \|r^{b+\frac{2}{q}-\frac{2}{p}} Y^\gamma v\|_{L_r^{p}([j-1,j+2])L_{\omega}^{2}} .\eneq
This is just the inequality \eqref{ap-S} if we set
$q={\frac{2p(p-1)}{p-2}}$ and $l^p$-sum over $j\ge R+1$ using the
Minkowski integral inequality.  Estimate \eqref{ap-Sinfty} follows
from obvious modifications of the same argument.

Inequality \eqref{ap-S4} follows from a similar argument.  The proof of \eqref{ap-trace-local} also yields
$$\|v\|_{L^4_r([j,j+1])L^4_\omega}\lesssim j^{-\frac{1}{2}}\sum_{|\gamma|\le 1}
\|Y^\gamma v\|_{L^2_r([j-1,j+2])L^2_\omega},
$$
which implies \eqref{ap-S4} after an application of H\"older's
inequality, weighting appropriately, and $l^4$-summing over $j$.
\end{prf}

\subsection{Weighted Strichartz estimates}
In this subsection, we prove inhomogeneous weighted Strichartz
estimates near spatial infinity. The exact form of the Kerr metric is
not important here; all that matters is that it is a small
perturbation of Minkowski.
We shall first prove an estimate for small perturbations of the
Minkowski space-time.  In the sequel, we shall then proceed to cutoff the Kerr
solution and focus on the exterior of a ball of sufficiently large radius that we may
view the Kerr metric as a small asymptotic perturbation of the Minkowski metric.

To this end, we set
\[\Box_h \phi = (\partial_t^2 -\Delta -\partial_\alpha h^{\alpha\beta}(t,x)\partial_\beta)
\phi,\]
where the summation convention is employed.  We
shall assume that
\begin{equation}
  \label{hdecay}
 h^{\alpha\beta} = h^{\beta\alpha},\quad |h|\le \frac{\delta }{\la
   x\ra^{\rho}},\quad |\partial h|\le \frac{\delta}{\la x\ra^{1+\rho}},
\end{equation}
for some $\rho>0$ and $\delta\ll 1$,
where in an abuse of notation we set
\[|h|=\sum_{\alpha,\beta=0}^3 |h^{\alpha\beta}(t,x)|,\quad |\partial
h|=\sum_{\alpha,\beta,\gamma=0}^3 |\partial_\gamma h^{\alpha\beta}(t,x)|.\]

Our main estimate near infinity is the following weighted Strichartz
estimate, which is closely akin to those first proved in \cite{HMSSZ}
and \cite{FaWa}.
\begin{thm}\label{thm-wStri}
Let $p\in [2,\infty)$.  Suppose $w$ solves
\[
\Box_h w=G_1+G_2, \qquad w(0,\cdot)=0=\pt w(0,\cdot),
\]
where $h$ satisfies \eqref{hdecay} for $\delta$ sufficiently small.
Additionally, suppose that $w$ vanishes in a neighborhood of the origin.
%
Then for any $\delta_1>0$ and $1/2-1/p<s<1/2$ we have
\beeq\label{weiStri}
 \|\la r\ra^{\frac{3}{2}-\frac{4}p-s} w\|_{L_{t, r}^{p} L^2_\omega}\lesssim
\|  r^{-\frac{1}{2}-s}  G_1\|_{ L^1_{t,r} L^2_\omega} + \|\la r\ra^{\frac{3}{2}-s+\delta_1}G_2\|_{L^2_{t,r,\omega}} .
 \eneq
\end{thm}

We will prove this by interpolating between the estimates that the
following two lemmas afford to us.  The first is a standard localized
energy estimate, while the second result is a variant of such where
we have divided through by a derivative in the spirit of, e.g.,
\cite[Lemma 2.3]{WaYu11}.

To begin, we note that the following localized energy estimate is an immediate corollary of the methods of
\cite{MetSo06}.  See also \cite{MetSo07} and \cite{MetTa07}.


\begin{lem} \label{lemmaKSSpert}Suppose that $h^{\alpha\beta}$ are smooth
  and satisfy \eqref{hdecay} for $\delta\ll 1$ sufficiently small.
  Let $G\in L^1_tL^2_x + l^{1/2}(L^2_t L^2_x)$, and
let $w$ solve $\Box_h w = G$ on $\R_+\times\R^3$.  Then
  \begin{equation}
    \label{ms_kss}
    \|\partial w\|_{L^\infty_tL^2_x\cap l^{-1/2}_\infty(L^2_tL^2_x)} +
    \|w\|_{l^{-3/2}_\infty(L^2_tL^2_x)} \lesssim \|\partial
    w(0,\cd)\|_2 + \|G\|_{L^1_tL^2_x+l^{1/2}_1(L^2_tL^2_x)}.
  \end{equation}
\end{lem}

The other endpoint for our real interpolation shall be:
\begin{lem}\label{thm-kss2}
Suppose that $G\in l^{1/2}_1(L^2_t\dot{H}^{-1}_x) +
L^1_t\dot{H}^{-1}_x$ and that
$w$ solves $\Box_h w = G$ with vanishing initial data.  Here $h$ is
assumed to satisfy \eqref{hdecay} for $\delta\ll 1$ sufficiently small.  Then
\begin{equation}
  \label{kss_noD}
  \|w\|_{l^{-1/2}_\infty(L^2_tL^2_x)\cap L^\infty_tL^2_x}\lesssim
  \|G\|_{l^{1/2}_1(L^2_t\dot{H}^{-1}_x)+L^1_t\dot{H}^{-1}_x}.
\end{equation}
\end{lem}

\begin{prf}
  Given a function $F\in l^{1/2}_1(L^2_tL^2_x)+L^1_tL^2_x$ over a
  $[0,T]$ time-strip, let $u$ solve $\Box_h u = F$ with vanishing data
  on the $t=T$ slice.  Applying Lemma \ref{lemmaKSSpert} backward in
  time, we obtain
\begin{equation}\label{dualLE} \|u\|_{l^{-1/2}_\infty(L^2_t\dot{H}^1_x) \cap L^\infty_t
  \dot{H}^1_x} \lesssim \|F\|_{l^{1/2}_1(L^2_tL^2_x)+L^1_tL^2_x}.\end{equation}

We then have
\[
\la w,F\ra = \la w, \Box_h u\ra = \la \Box_h w, u\ra \le
\|\Box_h w\|_{l^{1/2}_1(L^2_t\dot{H}^{-1}_x) + L^1_t\dot{H}^{-1}_x}
\|u\|_{l^{-1/2}_\infty(L^2_t\dot{H}^1_x)\cap L^\infty_t\dot{H}^1_x},
\]
where the inner product is that from $L^2_{t,x}([0,T]\times \R^3)$.
Applying \eqref{dualLE}, we get
\[\la w,F\ra \lesssim \|G\|_{l^{1/2}_1(L^2_t\dot{H}^{-1}_x)+
  L^1_t\dot{H}^{-1}_x} \|F\|_{l^{1/2}_1(L^2_tL^2_x)+L^1_tL^2_x},\]
which by duality completes the proof.
\end{prf}

We can now use real interpolation to complete the proof of Theorem \ref{thm-wStri}.  We record the following facts:
\begin{itemize}
\item (\cite[Theorem 5.6.1]{JJ})
  $[l_{q_0}^{s_0}(A),l_{q_1}^{s_1}(A)]_{\theta,q}=l_{q}^{(1-\theta)
  s_0+\theta s_1}(A)$ for any $s_0\neq s_1$, $0<q_0,q_1, q\le \infty$ and $\theta\in (0,1)$.
\item (\cite[Theorem 5.6.2]{JJ})
  $[l_{q_0}^{s_0}(A_0),l_{q_1}^{s_1}(A_1)]_{\theta,q}=l_{q}^{(1-\theta)
    s_0+\theta s_1}([A_0,A_1]_{\theta,q})$ if $0<q_0,q_1,q<\infty$ and
  $\frac 1q=(1-\theta)\frac{1}{q_0}+\theta\frac{1}{q_1}$.
\item (\cite[1.18.4]{TH})
  $[L^{p_0}(A_0),L^{p_1}(A_1)]_{\theta,p}=L^p([A_0,A_1]_{\theta,p})$ where
  $\frac 1p=\frac{1-\theta}{p_0}+\frac{\theta}{p_1}$, $1\le p_0,
  p_1<\infty$.
\item (\cite[Theorem 6.4.5]{JJ}) $[\dot H^{s_0},\dot H^{s_1}]_{\theta,2}=\dot H^{(1-\theta)s_0+\theta s_1}.$
\end{itemize}




We note that the left side of \eqref{ms_kss} controls
\[\|w\|_{l^{-1/2}_\infty(L^2_t\dot{H}^1_x)\cap
  l^{-3/2}_\infty(L^2_tL^2_x)} + \|w\|_{L^\infty_t \dot{H}^1_x}.\]
Thus, interpolation between \eqref{ms_kss} and \eqref{kss_noD} yields
\begin{equation}
  \label{5}
 \|w\|_{L^\infty_t \dot{H}^{s'}_x} + \|w\|_{l^{-1/2-s'}_2(L^2_tL^2_x)}
 \lesssim \|G_1\|_{L^1_t \dot{H}^{s'-1}_x} + \|G_2\|_{l^{1/2}_1(L^2_t\dot{H}^{s'-1}_x)}
\end{equation}
for any $0<s'<1$.  (Technically, for $i=1,2$, we are letting $w_i$ solve $\Box_h
w_i=G_i$ with vanishing initial data and doing separate interpolations
for each $i$.  We record this, however, as a single step.)  Combining
the above with the trace theorem on the sphere,
\begin{equation}\label{trace}
\|r^{\frac{3}{2}-s''} f\|_{L^\infty_r L^2_\omega}\lesssim \|f\|_{\dot H^{s''}},\quad 1/2<s''<3/2,
\end{equation}
we obtain
\beeq\label{6}
 \|r^{\frac{3}{2}-s''} w\|_{L^\infty_{t,r} L^2_\omega}\lesssim
 \|G_1\|_{L^1_t \dot H_x^{s''-1}} + \|G_2\|_{l^{1/2}_1(L^2_t\dot{H}^{s''-1}_x)},\quad 1/2<s''<1.
\eneq
Interpolating between the second term in \eqref{5} and \eqref{6} then
yields
\beeq\label{7}
 \|r^{\frac{3}{2}-\frac{4}{p}-s} w\|_{L^p_{t,r} L^2_\omega}\lesssim
 \|G_1\|_{L^1_t \dot H_x^{s-1}} + \|G_2\|_{l^{1/2}_1(L^2_t\dot{H}^{s-1}_x)},\quad \frac{1}{2}-\frac{1}{p}<s<1.
\eneq
If additionally $1-s>1/2$, then we can apply the dual to \eqref{trace}
to obtain
\beeq\label{8}
 \|r^{\frac{3}{2}-\frac{4}{p}-s} w\|_{L^p_{t,r}
   L^2_\omega}\lesssim \|r^{-\frac{1}{2}-s} G_1\|_{L^1_{t,r}
   L^2_\omega}
+ \|G_2\|_{l^{1/2}_1(L^2_t \dot{H}^{s-1}_x)},\quad  \frac{1}{2}-\frac{1}{p}<s<\frac{1}{2}.
\eneq
By duality, the Sobolev embedding $\dot H^{1-s} \subset L^{\frac{6}{2s+1}}$, and H\"{o}lder's inequality, we have
\beeq\label{9}
\|G_2\|_{l^{1/2}_1(L^2_t \dot{H}^{s-1}_x)}\lesssim
\|\<r\>^{\frac{3}{2}-s+\de}G_2\|_{L^2_{t,r,\omega}} ,\quad s\in (0,1/2),
\eneq
which completes the proof of Theorem \ref{thm-wStri}.






\section{The Strauss conjecture on Kerr black hole background}\label{sectionExistence}

We now prove our main theorem, Theorem \ref{metaTheorem}.  In fact,
with more notation in place, we first state a more precise version of
the theorem.

\begin{thm}\label{Glob-Kerr}
Suppose that the initial data $(f,g)\in H^3\times H^2$ have compact support.
 Then there exists a global solution $u$ in $\M$ for the problem \eqref{nlw} with $p>1+\sqrt{2}$, provided that
\beeq\label{Th-as}\|f\|_{H^3}+\|g\|_{H^2}=\varepsilon\ll 1\eneq is small enough.
Moreover, there is a large constant $R_0$, depending only on $M$ and
$a$, such that we have the following property: For any initial data
supported in a ball with radius $R$ with $R\ge R_0$ and any
$\delta>0$, there exists a constant $C>0$, depending on $F_p$, $R$, $\delta$, $M$ and $a$, so that we have the following estimate for the solution $u$,
\[
\sum_{|\gamma|\leq 2} \Bigl(
  \| r^{-\alpha} \chi Z^{\gamma} u\|_{L^q_\tv L^q_r
    L^2_\omega} + \|\la r\ra^{-3/2-\delta} \pa^\gamma u\|_{L^2_\tv
    L^2_r L^2_\omega}+ \|\pa^\gamma \partial u\|_{L^\infty_\tv
    L^2_r L^2_\omega}\Bigr)\le C \varepsilon.
\]
Here, $\chi(r)$ is a cutoff
function supported when $r>R$ so that $\chi = 1$ when $r>R+1$, and
$\alpha = \frac4q - \frac2{q-1}$ with $q=p$ if $p\in (1+\sqrt{2}, 3)$
and $q\in (1+\sqrt{2},3)$ if $p\ge 3$.
\end{thm}

In the proof that follows, we shall only focus on $p\in
(1+\sqrt{2},3)$.  This is, in part, because the cases $p>3$ have been
handled in previous work.  The adjustments to our proof needed to explore the
cases $p\ge 3$ are straightforward.  Indeed, you simply iterate in
the corresponding spaces for any index $q$ in the $(1+\sqrt{2},3)$ range
and use Sobolev embeddings to bound the $p-q$ extra copies of the
solution in the nonlinearity.




\subsection{Setting}
 We are interested in solving the equation
\eqref{nlw}
 \begin{equation}\label{maineqn}
\Box_K u = F_p(u), \qquad u |_{\Sigma^-} = f, \qquad \tilde T u |_{\Sigma^-} = g,
\end{equation}
 where we assume \eqref{key}.  The initial data are taken to have
 compact support and to be subject to \eqref{Th-as}.  We choose
 $R>3.5M$ sufficiently large so that the supports of $f$ and $g$ are
 contained within $\{r\le R\}$ and so that $(\rho^2/r^2)\Box_K$
 satisfies \eqref{hdecay} on $\{r\ge R\}$.

We let $\chi\in C^\infty (\R)$ satisfy $0\le \chi(r)\le 1$, $\chi(r)\equiv
0$ for $r\le R$, and $\chi(r)\equiv 1$ for $r>R+1$.
We shall utilize Theorem \ref{thm-wStri} with
$s=\frac{3}{2}-\frac{2}{p-1}$, which falls in the range $(1/2-1/p,1/2)$
precisely when $1+\sqrt{2}<p<3$.  For $\alpha =
\frac{4}{p}-\frac{2}{p-1} = \frac{2(p-2)}{p(p-1)}$, we define
\begin{equation}
   \|\phi\|_X =  \sum_{|\gamma|\leq 2} \Bigl(
  \| r^{-\alpha} \chi Z^{\gamma} \phi \|_{L^p L^p L^2} +
  \|\la r\ra^{-3/2-\delta} \partial^{\gamma} \phi\|_{L^2 L^2 L^2}
+\|\pa^\ga\pa \phi\|_{L^\infty L^2 L^2}\Bigr),\label{X-norm}
\end{equation}
\beeq\label{N-norm}
 \|g\|_N = \sum_{|\gamma|\leq 2} \Bigl(\| r^{- \alpha p} \chi^p Z^{\gamma} g \|_{L^1 L^1 L^2} + \|Z^{\gamma} g\|_{L^1 L^2 L^2}\Bigr)\ .
\eneq


\subsection{Main estimate}
In this section, we shall combine the weighted Strichartz estimates,
which are known to hold for small perturbations, and the localized
energy estimates, which are known to hold on Kerr provided $a\ll M$,
to prove our main estimate.  We shall also incorporate the necessary
vector fields and, as such, shall examine the associated commutators.

\begin{lem}\label{thm-key}
    Let $u$ be the solution to
\begin{equation}\label{lineqn}
  \Box_K u = G, \qquad u|_{\Sigma^-} = f, \qquad \tilde T u |_{\Sigma^-} = g.
 \end{equation}
 where $f$, $g$, $G(0,r\omega)$, and $\pa_\tv G(0,r\omega)$ are supported in the region $\{r\le  R\}$.
 Then there exists a constant $R_0$, such that for $R>R_0$, we have
\begin{equation}\label{linest}
\|u\|_X \lesssim \|f\|_{H^3}+\|g\|_{H^2}+ \|G\|_N \ .
\end{equation}
\end{lem}

\begin{proof}
  We first note that the latter two terms of \eqref{X-norm} are
  trivially controlled by the right side of \eqref{linest} using \eqref{main.estimate.inhom.n}.

Next, we record that
\[[\Box_K,\partial]  u = \CO(r^{-2}|\partial^2 u| + r^{-3}|\partial
u|)\]
and
\[[\Box_K,\Omega]u = \CO(a r^{-2} |\partial^2 u| + a r^{-3}|\partial
u|).\]
More generally,
\begin{equation}\label{vfComm}[\Box_K, \Omega^\alpha \partial^\beta] u =
\CO(r^{-2})\sum_{\substack{|\mu|+|\nu|\le |\alpha|+|\beta|\\|\mu|\le
    |\alpha|-1}} |\Omega^\mu \partial^\nu \partial u|
+ \CO(r^{-3}) \sum_{\substack{|\mu|+|\nu|\le |\alpha|+|\beta|\\|\mu|\le
    |\alpha|-1}} |\Omega^\mu \partial^\nu u|.
\end{equation}

Thus, we see that if $|\alpha|=1$, using \eqref{main.estimate.inhom.n} with $n=|\beta|$,
\begin{multline}
  \label{leDn}
 \|\partial^\beta \chi \Omega^\alpha  u\|_{LE}
   \lesssim \|\Omega^\alpha
  f\|_{H^{|\beta|+1}} + \|\Omega^\alpha g\|_{H^{|\beta|}} \\
  +
  \sum_{|\nu|\le |\beta|} \| \partial^\nu \Omega^\alpha G\|_{L^1L^2L^2}
  +  \sum_{|\nu|\le |\beta|}\|\pa^\ga [\Box_K,\chi] \Omega^\al
  u\|_{l^{1/2}_1(L^2L^2L^2)} \\+ \sum_{|\nu| \le |\beta|} \|\partial^\nu \chi
  [\Box_K, \Omega^\alpha] u\|_{l^{1/2}_1(L^2L^2L^2)}.
\end{multline}
As $[\Box_K,\chi]$ is compactly supported and the coefficients of $\Omega$ are $\CO(1)$ on
the support of $[\Box_K,\chi]$, we have
\[ \sum_{|\nu|\le |\beta|}\|\pa^\ga [\Box_K,\chi] \Omega^\al
  u\|_{l^{1/2}_1(L^2L^2L^2)} \lesssim \sum_{|\nu|\le |\beta|+1}
\|\partial^\nu u\|_{LE}.\]
Similarly, using the commutator estimate above,
\[\sum_{|\nu|\le |\beta|} \|\partial^\nu \chi [\Box_K,\Omega^\alpha]
u\|_{l^{1/2}_1(L^2L^2L^2)} \lesssim \sum_{|\nu|\le |\beta|+1}
\|\partial^\nu u\|_{LE}.\]
  If we, in turn, apply
\eqref{main.estimate.inhom.n} with $n=|\beta|+1$, it follows that
\begin{equation}
  \label{leDn2}
  \sum_{\substack{|\alpha|+|\beta|\le 2\\|\alpha|\le 1}} \| \partial^\beta \chi Z^\alpha u\|_{LE}
  \lesssim \|f\|_{H^3} + \|g\|_{H^2} +
  \sum_{|\alpha|\le 2} \|Z^\alpha G\|_{L^1L^2L^2}.
\end{equation}
Higher order estimates akin to this have previously appeared in, e.g., \cite{MTT}.

We now turn to bounding the first term in \eqref{X-norm}.  We note
that
\[\Box_K \chi Z^\gamma u = \chi Z^\gamma G + [\Box_K,\chi] Z^\gamma u
+ \chi[\Box_K,Z^\gamma]u.\]
We also note that $(\rho^2/r^2)\Box_K$ satisfies the requirements
\eqref{hdecay} on the support of $\chi$ when $R$ is sufficiently large.  The support conditions on
$f$, $g$, and $G$ guarantee that the Cauchy data for $\chi Z^\gamma u$
vanish.  Using that $\rho^2/r^2$ is $\CO(1)$ on the support of $\chi$,
it follows from \eqref{weiStri} that
\begin{multline}\label{rhs}\sum_{|\gamma|\le 2} \|r^{-\alpha} \chi Z^\gamma u\|_{L^pL^pL^2}\lesssim
\sum_{|\gamma|\le 2} \|r^{-\alpha p} \chi Z^\gamma G\|_{L^1L^1L^2} \\+
\sum_{|\gamma|\le 2} \|r^{\frac{3}{2}-s+\delta}[\Box_K,\chi]
Z^\gamma u\|_{L^2L^2L^2} +
\sum_{|\gamma|\le 2} \|r^{\frac{3}{2}-s+\delta}\chi [\Box_K,Z^\gamma] u\|_{L^2L^2L^2}.
\end{multline}
We first note that
\[\|r^{-\alpha p} \chi Z^\gamma G\|_{L^1 L^1 L^2} \lesssim
\|r^{-\alpha p} \chi^p Z^\gamma G\|_{L^1 L^1 L^2} + \|Z^\gamma
G\|_{L^1 L^2L^2}\]
since $\chi-\chi^p$ is supported where $r\in [R,R+1]$.
As the weight in the second term in the right side of \eqref{rhs} is $\CO(1)$ on the support of $[\Box_K,\chi]$ and as
that support is contained in $\{r\ge 3.5M\}$, it follows that the
second term on the right is bounded by $\sum_{{\gamma}\le 2}
\|\partial^\gamma u\|_{LE}$ and can be controlled using
\eqref{main.estimate.inhom.n}.  Similarly, we can choose $0<\de<s$ and use \eqref{vfComm} to control the last
term on the right by
\[\sum_{\substack{|\gamma|+|\mu|\le 2\\|\gamma|\le 1}}
\|\partial^\mu \chi Z^\gamma  u\|_{LE}+\sum_{\substack{|\mu|\le 2}}
\|\partial^\mu u\|_{LE},\]
for which \eqref{leDn2} and \eqref{main.estimate.inhom.n} provide the desired bound.
\end{proof}

\subsection{The Strauss conjecture}
We can now prove Theorem \ref{Glob-Kerr}.  We solve \eqref{maineqn}
via iteration.  We set $u_0\equiv 0$ and recursively define $u_{k+1}$
to be the solution to the linear equation
\beeq\label{3iterate}
  \Box_K u_{k+1} = F_p(u_{k}), \qquad   u|_{\Sigma^-} = f, \qquad
  \tilde T u|_{\Sigma^-} = g.
\eneq

{\em Boundedness:}
By the smallness condition \eqref{Th-as} on the data as well as the
condition imposed on their supports, it follows from Lemma
\ref{thm-key} that there is a universal constant $C_1$ so that
\[\|u_1\|_X \le C_1\varepsilon\]
and
\[\|u_{k+1}\|_X \le C_1\varepsilon + C_1 \|F_p(u_k)\|_N.\]

We shall argue inductively to prove that
\[\|u_{k+1}\|_X \le 2C_1 \varepsilon.\]
By the above, it suffices to show
\[\|F_p(u_k)\|_N\le \varepsilon.\]
By condition \eqref{key}, we have
\[|Z^\gamma F_p(u_k)|\lesssim |u_k|^{p-1} |Z^\gamma u_k| +
|u_k|^{p-2}\Bigl(\sum_{|\beta|\le 1} |Z^\beta u_k|\Bigr)^2\]
for $|\gamma|\le 2$.

We start with bounding the first term in \eqref{N-norm}.  We first
note that
\begin{multline*}\sum_{|\gamma|\le 2} \| r^{-\alpha p} \chi^p Z^\gamma
F_p(u_k)\|_{L^1L^1L^2} \lesssim
\|r^{-\alpha} \chi u_k\|^{p-1}_{L^pL^pL^\infty} \sum_{|\gamma|\le 2}
\|r^{-\alpha} \chi Z^\gamma u_k\|_{L^pL^pL^2}
\\+ \|r^{-\alpha} \chi u_k\|^{p-2}_{L^pL^pL^\infty} \Bigl(\sum_{|\beta|\le 1}
\|r^{-\alpha} \chi Z^\beta u_k\|_{L^pL^pL^4}\Bigr)^2.
\end{multline*}
By the $H^2_\omega\subset L^\infty_\omega$ and $H^1_\omega\subset
L^4_\omega$ Sobolev embeddings on $\S^2$, it follows that the right
side above is $\CO(\|u_k\|^p_X)$.

We now proceed to the second term in \eqref{N-norm}.  We first observe
that
\[\sum_{|\gamma|\le 2} \|u_k^{p-1} Z^\gamma u_k\|_{L^1L^2L^2(r\ge
  R+2)}
\lesssim \|r^{\frac{\alpha}{p-1}}
u_k\|^{p-1}_{L^pL^{\frac{2p(p-1)}{p-2}}L^\infty(r\ge R+2)} \sum_{|\gamma|\le 2}
\|r^{-\alpha} \chi Z^\gamma u_k\|_{L^pL^pL^2}\]
Applying \eqref{ap-S}, it follows that the right side is
\[\lesssim\Bigl(\sum_{|\gamma|\le 2} \|r^{\frac{\alpha-1}{p-1}} \chi Z^\gamma
  u_k\|_{L^pL^pL^2}\Bigr)^{p-1} \sum_{|\gamma|\le 2} \|r^{-\alpha}
  \chi Z^\gamma u_k\|_{L^pL^pL^2}.\]
As $\frac{\alpha-1}{p-1}\le -\alpha$ for $p\le 3$, it follows that
this is also $\CO(\|u_k\|_X^p)$.  Moreover,
\begin{multline*}\sum_{|\gamma|\le 2} \|u_k^{p-1} Z^\gamma u_k\|_{L^1L^2L^2(r\le
  R+2)} \\\lesssim
\|u_k\|^{p-2}_{L^\infty L^\infty L^\infty(r\le R+2)}
\|u_k\|_{L^2L^\infty L^\infty(r\le R+2)} \sum_{|\gamma|\le 2}
\|\partial^\gamma u_k\|_{L^2 L^2L^2(r\le R+2)}.
\end{multline*}
Sobolev embeddings allow us to control this by
\begin{equation}\label{locTerm}\Bigl(\sum_{|\gamma|\le 2} \|\partial^\gamma \partial
u_k\|_{L^\infty L^2 L^2}\Bigr)^{p-2} \Bigl(\sum_{|\gamma|\le 2}
\|\partial^\gamma u_k\|_{L^2L^2L^2(r\le R+3)}\Bigr)^2,\end{equation}
which is also $\CO(\|u_k\|_X^p)$.


We similarly examine
\begin{multline*}\Bigl\|u^{p-2}_k \sum_{|\beta|\le 1} Z^\beta u_k\Bigr\|_{L^1L^2L^2(r\ge
  R+2)} \\\lesssim \|r^{\frac{1}{p}} u_k\|^{p-2}_{L^p L^\infty
    L^\infty(r\ge R+2)} \Bigl(\sum_{|\beta|\le 1} \|r^{-\frac{1}{p} +
    \frac{4-p}{2p}} Z^\beta u_k\|_{L^p L^4 L^4(r\ge R+2)}\Bigr)^2.
\end{multline*}
Applications of \eqref{ap-S4} and \eqref{ap-Sinfty}, it follows that
this is bounded by
\[\Bigl(\sum_{|\gamma|\le 2} \|r^{-\frac{1}{p}} \chi Z^\gamma
u_k\|_{L^pL^pL^2}\Bigr)^p.\]
As $\alpha\le 1/p$ for $p\le 3$, we have that these terms are
$\CO(\|u_k\|_X^p)$.  It remains to bound
\[\Bigl\|u^{p-2}_k \sum_{|\beta|\le 1} Z^\beta u_k\Bigr\|_{L^1L^2L^2(r\le
  R+2)} \lesssim \|u_k\|^{p-2}_{L^\infty L^\infty
    L^\infty(r\le R+2)} \Bigl(\sum_{|\beta|\le 1} \|Z^\beta u_k\|_{L^2
    L^4 L^4(r\le R+2)}\Bigr)^2.\]
Using Sobolev embeddings, this is estimated by \eqref{locTerm}, which
as noted above is $\CO(\|u_k\|_X^p)$.

Combining the above, we have
\[\|F_p(u_k)\|_N\le C \|u_k\|_X^p \le C (C_1 \varepsilon)^p,\]
where we have employed the inductive hypothesis.  As long as
$\varepsilon$ is chosen sufficiently small that $C C_1^p
\varepsilon^{p-1}\le 1$, the proof of boundedness is complete.

{\em Convergence of the sequence $\{u_k\}$:}  We shall complete the
proof by showing that the sequence $\{u_k\}$ is Cauchy in $X$.  For
$k\ge 1$, we have
\[\|u_{k+1}-u_k\|_X \le C_1 \|F_p(u_k)-F_p(u_{k-1})\|_N.\]
Mimicking the proof above shows that
\begin{align*}\|F_p(u_k)-F_p(u_{k-1})\|_N &\le C(\|u_k\|^{p-1}_X +
\|u_{k-1}\|_X^{p-1}) \|u_k-u_{k-1}\|_X \\&\le 2C C_1^{p-1}
\varepsilon^{p-1} \|u_k-u_{k-1}\|_X.
\end{align*}
For all $\varepsilon$ sufficiently small, this implies that
\[\|u_{k+1}-u_k\|_X \le \frac{1}{2}\|u_k-u_{k-1}\|_X,\]
which suffices to show that the sequence is Cauchy, and hence
completes the proof of Theorem \ref{Glob-Kerr}.



\section{The Strauss conjecture on the Schwarzschild background}\label{SchwSection}

As pointed out before, the technical assumption of compactly support for the initial data can be removed  in the case of the Schwarzschild background ($a=0$) by adapting the arguments in \cite{SoWa10} and \cite{WaYu11}. 

Consider the evolution of the nonlinear waves in the cylindrical region $\M$,
\begin{equation}\label{ap-nlw2}
\Box_S u = F_p(u), \qquad u_{|\Sigma^-} = f, \qquad \Tilde T u_{|\Sigma^-} = g\ .
\end{equation}
We have the following theorem, which is analogous to Theorem\ref{Glob-Kerr}:
\begin{thm}\label{ap-Glob-Schw}
Let $p>1+\sqrt{2}$.
 Then there exists a global solution $u$ in $\M$ for the problem \eqref{ap-nlw2}, provided that the initial data $(f,g)$ satisfy
$$
 E[f, g] := \sum_{|\gamma|\leq 3} \|Y^{\gamma} f\|_{\dot H^s \cap L^2} + \sum_{|\gamma|\leq 2} \|Y^{\gamma} g\|_{\dot H^{s-1} \cap L^2}
< \ep\ll 1$$
for $s=s(p)= \frac32 - \frac{2}{q-1}$.
Moreover, there is a large constant $R_0$, depending only on $M$, such that we have the following property: 
for any $\delta > 0$, there exists a $C>0$, depending on $F_p$, $R\ge R_0$, $\delta$ and $M$, so that  we have the following estimate for $u$:
\[
\sum_{|\gamma|\leq 2} \Bigl(
  \| r^{-\alpha} \chi Y^{\gamma} u\|_{L^q_\tv L^q_r
    L^2_\omega} + \|\la r\ra^{-3/2-\delta} \nabla_x^\gamma u\|_{L^2_\tv
    L^2_r L^2_\omega}+ \|\nabla_x^\gamma \partial u\|_{L^\infty_\tv
    L^2_r L^2_\omega}\Bigr)\le C \varepsilon.
\]
Here $q=p$ if $p\in (1+\sqrt{2}, 3)$ and $q\in (1+\sqrt{2},3)$ if $p\ge 3$,
$\alpha = \frac4q - \frac2{q-1}$, and $\chi(r)$ is a cutoff
function supported when $r>R$ so that $\chi = 1$ when $r>R+1$.
\end{thm}

\begin{rem}\label{ap-rem2}
The key fact for us to prove Theorem \ref{ap-Glob-Schw} is that we can rewrite the D'Alembertian as $\pt^2+P$ for a certain self-adjoint time-independent Laplacian $P$.
\end{rem}

As in the Kerr case above, we shall focus only on $p\in
(1+\sqrt{2},3)$ and will drop the $q$ notation.

The only important change in the argument is the following version of the weighted Strichartz estimates, which replaces Theorem \ref{thm-wStri}
\begin{prop}[Weighted Strichartz estimates for Schwarzschild]\label{ap-thm-le-0wS}

Let $u$ solve the equation $\Box_S u=G_1+G_2$ with initial data $(f, g)$ on $\{\tv=0\}$. Moreover, assume that $u$ vanishes in the region $\{r<K M\}$ for some large number $K>0$. Then for any $p\ge 2$, $1/2-1/p<s<1/2$ and $\de>0$, we have
\beeq\label{ap-eq-le-0wS}
\begin{array}{ll}
  \|r^{3/2-4/p-s} u\|_{L^p_{\tv,r}L^2_\omega(\M)} & \lesssim  \|f\|_{\dot H^s}+\|g\|_{\dot H^{s-1}}\\
  &+\|r^{-1/2-s}G_1\|_{L^1_{\tv,r} L^2_\omega}+\|r^{3/2-s+\de}G_2\|_{L^2_{\tv,x}}.
\end{array}
\eneq
\end{prop}

 Let us now prove the Strauss conjecture, Theorem \ref{ap-Glob-Schw}, based on this Proposition.

 We choose $R = KM$, where $K$ is large enough so that Proposition \ref{ap-thm-le-0wS} holds. We let $\chi\in C^\infty(\R)$ satisfy $0\le \chi(r)\le 1$, $\chi(r)\equiv
0$ for $r\le R$, and $\chi(r)\equiv 1$ for $r>R+1$. For $\alpha = \frac{4}{p}-\frac{2}{p-1}$, we define
\begin{equation}
   \|\phi\|_X =  \sum_{|\gamma|\leq 2} \Bigl(
  \| r^{-\alpha} \chi Y^{\gamma} \phi \|_{L^p L^p L^2} +
  \|\la r\ra^{-3/2-\delta} \nabla_x^{\gamma} \phi\|_{L^2 L^2 L^2}
+\|\nabla_x^\ga\pa \phi\|_{L^\infty L^2 L^2}\Bigr),\label{Sc-X-norm}
\end{equation}
\beeq\label{Sc-N-norm}
 \|g\|_N = \sum_{|\gamma|\leq 2} \Bigl(\| r^{- \alpha p} \chi^p Y^{\gamma} g \|_{L^1 L^1 L^2} + \|Y^{\gamma} g\|_{L^1 L^2 L^2}\Bigr)\ ,
\eneq
where $\delta>0$ is arbitrarily fixed.

We can now utilize Proposition \ref{ap-thm-le-0wS} with
$s=\frac{3}{2}-\frac{2}{p-1}$, which falls in the range $(0,1/2)$
precisely when $1+\sqrt{2}<p<3$, to prove the equivalent of Lemma \ref{thm-key}:

 \begin{lem}\label{ap-key}
  Let $u$ be the solution to
\begin{equation}\label{ap-lineqn}
  \Box_S u = F, \qquad u_{|\Sigma^-} = f, \qquad \pa_{\tv} u_{|\Sigma^-} = g\ .
 \end{equation}
 Then we have for $1+\sqrt{2}<p<3$
\begin{equation}\label{ap-linest}
\|u\|_X \lesssim \|F\|_N + E[f, g]\ .
\end{equation}
 \end{lem}

 The proof is similar (but easier) to that of Lemma \ref{thm-key} , where one uses \eqref{ap-eq-le-0wS} instead of \eqref{weiStri} to bound the commutator term $\chi[\Box_S, Z^{\gamma}]u$. With Lemma \ref{ap-key} instead of Lemma \ref{thm-key}, it is easy to see that the proof of the Strauss conjecture, Theorem \ref{ap-Glob-Schw}, is similar to that of Theorem \ref{Glob-Kerr}.

\subsection{Proof of Proposition \ref{ap-thm-le-0wS}}

As in \cite{SoWa10}, we want to rewrite the equation near infinity as $(\pt^2+P) w=F$ so that $P$ is elliptic and self-adjoint Laplacian with respect to $L^2(\R^3)$.


Recall that, in the $(t,r,\omega)$ coordinates, we have
$$\Box_S=-\weight^{-1}\pt^2+r^{-2}\pa_r r^2 \weight \pa_r+r^{-2}\Delta_\omega=-\weight^{-1}(\pt^2+Q)$$
where $-Q=\weight r^{-2}\pa_r r^2 \weight \pa_r+\weight r^{-2}\Delta_\omega$ is a self-adjoint operator with respect to the metric $\weight^{-1}dx$. We see that $\Box_S u=G$ is equivalent to
$$(\pt^2+Q) u=-\weight G\ .$$

The self-adjoint operator we are seeking is
\[
P=\weight^{-1/2} Q \weight^{1/2}
\]

In conclusion, the above calculation tells us that
$$\Box_S u=G$$ is equivalent to the equation
\beeq\label{ap-NLW2}(\pt^2+P) w=- \weight^{1/2} G\ ,\eneq
with  $w=\weight^{-1/2} u$ and $P$ as above,
\[
 P = -\weight^{1/2} r^{-2}\pa_r r^2 \weight \pa_r  \weight^{1/2}-\weight r^{-2}\Delta_\omega.
\]


Since $u$ is supported in $r \geq K M$ (where we have $\tv=t$), we can essentially reduce the problem to the problems studied in \cite{SoWa10} and \cite{WaYu11}. Since $w$ vanishes for $r<KM$, we can easily extend $P$ to a self-adjoint operator $P_1$ on $\R^3$ (e.g.,
\beeq\label{ap-self-adjoint2}
P_1=-h r^{-2}\pa_r  r^2 h^2 \pa_r  h
-h^2 r^{-2}\Delta_\omega\eneq
where $h(r)=(1-\tilde\psi)\weight^{1/2}+\tilde\psi$ and $\tilde\psi$ is a radial function vanishing for $r>KM$ with $\tilde\psi=1$ for $r<KM/2$).

In particular, we note that $P_1^{1/2}\in S_{hom}^1$ has symbol
\[
 p_1 = h^2 \sqrt{p} + \frac1r e
\]
 Here $\sqrt{p}$ is the symbol of $\sqrt{-\Delta}$ and $e\in S_{hom}^0$, $e\equiv 0$ for $r<KM/2$.

We will need the following lemma, which asserts that $P_1^{1/2}$ behaves like $\nabla_x$ in the appropriate function spaces:
 \begin{lem}\label{P1est}
 The following estimates hold in $\R^3$:
\begin{equation}\label{P1weight}
 \|P_1^{1/2} v\|_{l^{-1/2}_\infty(L^2)} \lesssim  \|\nabla_x v\|_{l^{-1/2}_\infty(L^2)} + \|v\|_{l^{-3/2}_\infty(L^2)}
\end{equation}
\begin{equation}\label{P1H1}
 \|P^{s/2}_1 u\|_{L^2} \simeq \|u\|_{\dot{H}^s}, \qquad s\in[-1,1].
\end{equation}
 \end{lem}
\begin{proof}
  Let $K_e$ denote the kernel of the operator associated to the error $\frac1r e$. For \eqref{P1weight} it is enough to prove that
\[
 \|\int K_e (x,y) v(y)dy \|_{l^{-1/2}_\infty(L^2)} \lesssim \|v\|_{l^{-3/2}_\infty(L^2)}
\]

 Fix a dyadic region $A_k = \{2^{k-1}\leq |x| \leq 2^k\}$. When $|y|\approx 2^k$, we easily get the bound
\begin{equation}\label{xeqy}
 \|\int_{2^{k-2}\leq |y|\leq 2^{k+1}} K_e (x,y) v(y)dy\|_{L^2(A_k)} \lesssim \|r^{-1} v\|_{L^2(2^{k-2}\leq r \leq 2^{k+1})}\lesssim 2^{k/2} \|v\|_{l^{-3/2}_\infty(L^2)}
\end{equation}

 When $|y-x| \gg 1$, we use the bound
 \[
 |K_e (x, y)|\lesssim |x-y|^{-n}|x+y|^{-1}, \qquad |x \pm y|\geq 1
 \]
 which comes from $e\in S_{hom}^0$ and the self-adjointess of the operator.

 By H\"older's inequality we get for $j>k+2$:
\[
 |\int_{2^{j-1} \leq |y|\leq 2^j} K_e (x,y) v(y)dy| \lesssim |\int_{2^{j-1} \leq |y|\leq 2^j} 2^{-4j} v(y)dy| \lesssim 2^{-j} \|v\|_{l^{-3/2}_\infty(L^2)}
\]

Similarly for $j<k-1$ we obtain the bound
\[
 |\int_{2^{j-1} \leq |y|\leq 2^j} K_e (x,y) v(y)dy| \lesssim 2^{3j-4k} \|v\|_{l^{-3/2}_\infty(L^2)}
\]

 Thus after summation we obtain the pointwise estimate
\[
 |\int_{|y-x| \gg 1} K_e (x,y) v(y)dy| \lesssim 2^{-k} \|v\|_{l^{-3/2}_\infty(L^2)}
\]
from which \eqref{P1weight} follows immediately, taking also \eqref{xeqy} into account.

By interpolation and duality, we
need only to prove the estimate \eqref{P1H1} for the special case where $s=1$.

For $h(r)=(1-\tilde\psi)\weight^{1/2}+\tilde\psi$, by choosing $K$ large enough , we have
\beeq\label{ap-f} 1/2\le h\le 2\ , |h'|\le C/r\ , |\pa_r h^{-1}|\le C/r\ .\eneq
By the expression of $P_1$ \eqref{ap-self-adjoint2}, it is easy to see that
$$  \|P_1^{1/2} u\|_{L^2}^2
=
  \int_{\R^3}h^2\left(
|\pa_r  h u|^2+\left|\ang u
\right|^2\right) dx\simeq \|\nabla (h u)\|_{L^2}^2\ .$$
However, the equivalence between $\|\nabla u\|_{L^2}$ and $\|\nabla(h u)\|_{L^2}$ can be seen
 by the Hardy inequality and \eqref{ap-f}.

 \end{proof}
 The proof of Proposition \ref{ap-thm-le-0wS} is reduced to the following proposition.
\begin{prop}\label{ap-thm-le-0wS-1}
Let $w$ solves the equation $(\pt^2+P_1) w=G=G_1+G_2$ with initial data $(w_0, w_1)$ on $\{t=0\}$, we have
\beeq\label{ap-eq-le-0wS-1}
\begin{array}{ll}
  \|r^{3/2-4/q-s} w\|_{L^q_{t,r\ge 1}L^2_\omega} & \le  C (\| w_0\|_{\dot H^s}+\|w_1\|_{\dot H^{s-1}}\\
  &+\|r^{-1/2-s}G_1\|_{L^1_{t,r} L^2_\omega}+\|\<r\>^{3/2-s+\de}G_2\|_{L^2_{t,x}}),
\end{array}
\eneq for any $q\ge 2$, $1/2-1/q<s<1/2$.
\end{prop}

 Let us give the proof of Proposition \ref{ap-thm-le-0wS}, based on Proposition \ref{ap-thm-le-0wS-1}.\\
 {\noindent {\bf Proof of Proposition \ref{ap-thm-le-0wS}.} }
 Since $u$ vanishes in the region $\{r<KM\}$, $\Box_S u=F$ with initial data $(u_0, u_1)$ is equivalent to
$$(\pt^2+P_1) w=G$$
with $w=\weight^{-1/2} u$ with initial data $(w_0,w_1)=(\weight^{-1/2} f, \weight^{-1/2}g)$ and $G=-\weight^{1/2}F$. Noting the support property of $u$, there is a cutoff function $\phi$ with support in $\{r<KM\}$ and $\phi=1$ for $r<(K-1)M$, such that $u=(1-\phi)u$ and so
$$w=\weight^{-1/2} (1-\phi) u\ .$$
Then by the fractional Leibniz rule (see e.g. Lemma 2.7 of \cite{WaYu11}), we have
\beeq\label{ap-Leib}\|w\|_{\dot H^s}\le C\left\|\weight^{-1/2} (1-\phi)\right\|_{L^\infty\cap \dot W^{1,3}} \|u\|_{\dot H^s}\le C \|u\|_{\dot H^s}, s\in [-1,1]\ .\eneq
Thus, by Proposition \ref{ap-thm-le-0wS-1}, we get
 \begin{eqnarray*}
  \|r^{3/2-4/q-s} u\|_{L^q_{\tv,r}L^2_\omega(\M)} & \le &  \|r^{3/2-4/q-s} w\|_{L^q_{t,r\ge 1}L^2_\omega}\\
   & \le & C (\| w_0\|_{\dot H^s}+\|w_1\|_{\dot H^{s-1}}\\
  &&+\|r^{-1/2-s}G_1\|_{L^1_{t,r} L^2_\omega}+\|\<r\>^{3/2-s+\de}G_2\|_{L^2_{t,x}})\\
&\le&   C (\| f\|_{\dot H^s}+\|g\|_{\dot H^{s-1}}\\
&  &+\|r^{-1/2-s}F_1\|_{L^1_{t,r} L^2_\omega}+\|\<r\>^{3/2-s+\de}F_2\|_{L^2_{t,x}}),
    \end{eqnarray*} where we have used the inequality \eqref{ap-Leib} for $w_0$ and $w_1$.  This completes the proof.
 {\hfill  {\vrule height6pt width6pt depth0pt}\medskip}

 Now we turn to the proof of the Proposition \ref{ap-thm-le-0wS-1}. Let $w_{hom}$ be the solution to the homogeneous problem with initial data $(w_0, w_1)$.
Inequality \eqref{main.estimate.inhom} implies in particular that
\begin{equation}\label{SchwH1}
  \|\nabla w_{hom}\|_{L^\infty_tL^2_x\cap l^{-1/2}_\infty(L^2_tL^2_x)} + \|w_{hom}\|_{l^{-3/2}_\infty(L^2_tL^2_x)} \lesssim \| w_0\|_{\dot{H}^1} + \|w_1\|_{L^2}.
\end{equation}

 Since $\partial_t$ and $P_1$ commute with $P_1^{-1/2}$, we obtain
\[
\|\nabla P_1^{-1/2} w_{hom}\|_{L^\infty_t L^2_x\cap l^{-1/2}_\infty(L^2_tL^2_x)} + \| P_1^{-1/2} w_{hom}\|_{l^{-3/2}_\infty(L^2_tL^2_x)} \lesssim \|P_1^{-1/2} w_0\|_{\dot{H}^1} + \|P_1^{-1/2} w_1\|_{L^2}
\]
which after using Lemma \ref{P1est} yields
\begin{equation}\label{SchwL2}
 \|w_{hom}\|_{L^\infty_tL^2_x\cap l^{-1/2}_\infty(L^2_tL^2_x)} \lesssim \|w_0\|_{L^2} + \|w_1\|_{\dot{H}^{-1}}.
\end{equation}

 The inhomogeneous part $w-w_{hom}$ has vanishing initial data, so we can use the estimates \eqref{ms_kss} and \eqref{kss_noD} to bound it. We thus obtain, using \eqref{SchwH1} and \eqref{SchwL2}:
\[
 \|\nabla w\|_{L^\infty_tL^2_x\cap l^{-1/2}_\infty(L^2_tL^2_x)} \lesssim \| w_0\|_{\dot H^1}+\|w_1\|_{L^2} + \|G\|_{L^1_tL^2_x + l^{1/2}_1(L^2_t L^2_x)}
\]
\[
\| w\|_{L^\infty_tL^2_x\cap l^{-1/2}_\infty(L^2_tL^2_x)} \lesssim \| w_0\|_{L^2}+\|w_1\|_{\dot H^{-1}} + \|G\|_{l^{1/2}_1(L^2_t\dot{H}^{-1}_x) +
L^1_t\dot{H}^{-1}_x}
\]
 The proof of Proposition \ref{ap-thm-le-0wS-1} now follows as in Proposition \ref{thm-wStri} through interpolation and the use of the trace lemma.

\end{document}